\documentclass[11pt]{article}
\usepackage{color,latexsym,amsthm,amsmath,amssymb,path,enumitem,url}

\newcommand{\ignore}[1]{}

\newtheorem{theorem}{Theorem}[section]
\newtheorem{lemma}[theorem]{Lemma}

\newtheorem{claim}[theorem]{Claim}

\newcommand{\Proof}[1]
        {
        \noindent
        \emph{Proof #1.}~
        }
\newsavebox{\smallProofsym}                     
\savebox{\smallProofsym}                        %
        {
        \begin{picture}(7,7)                    %
        \put(0,0){\framebox(6,6){}}             %
        \put(0,2){\framebox(4,4){}}             %
        \end{picture}                           %
        }                                       %
\newcommand{\smalleop}[1]
        {
        \mbox{} \hfill #1~~\usebox{\smallProofsym}\!\!\!\!\!\!\
        }

\newcommand{\parag}[1]{\vspace{2mm}

\noindent{\bf #1} }

\usepackage{graphicx,psfrag}
\usepackage[rflt]{floatflt}

\usepackage{dsfont}

\newcommand{\ZZ}{\ensuremath{\mathbb Z}}
\newcommand{\RR}{\ensuremath{\mathbb R}}

\def\eps{{\varepsilon}}

\begin{document}
\pagenumbering{arabic}

\title{A Construction for Difference Sets with Local Properties\footnote{This research project was done as part of the 2018 CUNY Combinatorics REU, supported by NSF grant DMS-1710305.}}

\author{
Sara Fish\thanks{California Institute of Technology, Pasadena, CA, USA.
{\sl sfish@caltech.edu}. Supported by Caltech's Summer Undergraduate Research
Fellowships (SURF) program.}
\and
Ben Lund\thanks{Department of Mathematics, Princeton University, Princeton, NJ, USA.
{\sl lund.ben@gmail.com}. Supported by NSF grant 1802787.}
\and
Adam Sheffer\thanks{Department of Mathematics, Baruch College, City University of New York, NY, USA.
{\sl adamsh@gmail.com}. Supported by NSF award DMS-1710305 and PSC-CUNY award 61666-00-49.}}

\date{}
\maketitle
\begin{abstract}
We construct finite sets of real numbers that have a small difference set and strong local properties.
In particular, we construct a set $A$ of $n$ real numbers such that $|A-A|=n^{\log_2 3}$ and that every subset $A'\subseteq A$ of size $k$ satisfies $|A'-A'|\ge k^{\log_2 3}$.
This construction leads to the first non-trivial upper bound for the problem of distinct distances with local properties.
\end{abstract}

\section{Introduction}

The \emph{Erd\H os distinct distances problem} is a main problem in Discrete Geometry.
This problem asks for the minimum number of distinct distances spanned by any $n$ points in $\RR^2$.
In 1946, Erd\H os \cite{erd46} observed that a $\sqrt{n}\times \sqrt{n}$ section of the integer lattice $\ZZ^2$ spans $\Theta(n/\sqrt{\log n})$ distinct distances.
He then conjectured that no set of $n$ points spans an asymptotically smaller number of distances.
Proving that every $n$ points determine at least some number of distinct distances turned out to be a deep and challenging problem.
After over 60 years and many papers on the problem, Guth and Katz \cite{GK15} proved that every set of $n$ points in $\RR^2$ spans $\Omega(n/ \log n)$ distances, almost completely settling Erd\H os's conjecture.

In the current work we study a variant of the distinct distances problem, also posed by Erd\H os (for example, see \cite{Erdos86}).
Even after Guth and Katz introduced sophisticated tools for studying distinct distances, this problem remains wide open.
For parameters $n,k,\ell$, we consider sets of $n$ points in $\RR^2$ with the property that every $k$ of the points span at least $\ell$ distinct distances.
Let $\phi(n,k,\ell)$ denote the minimum number of distinct distances such a point set can span.
That is, by having a local property of every small subset of points, we wish to obtain a global property of the entire point set.

As a first example, consider $\phi(n,3,3)$.
This is the minimum number of distinct distances spanned by $n$ points that do not span any isosceles triangles (including degenerate triangles with three collinear vertices).
Since no isosceles triangles are allowed, every point spans $n-1$ distinct distances with the other points, implying that $\phi(n,3,3) \ge n-1$.
Erd\H os \cite{Erdos86} derived the bound $\phi(n,3,3) < n2^{O(\sqrt{\log n})}$, leaving a subpolynomial gap in this case.

As another simple example, for any $k\ge 4$ we have
\[ \phi\left(n,k,\binom{k}{2}-\left\lfloor k/2 \right\rfloor +2 \right) = \Omega\left(n^2\right). \]
Indeed, if some distance occurs $\lfloor k/2 \rfloor$ times, then the corresponding points form a set of at most $k$ points with at most $\binom{k}{2}-\left\lfloor k/2 \right\rfloor +1$ distances.
This violates the local condition, so in this case every distance occurs at most $\lfloor k/2 \rfloor -1$ times.

A result of Erd\H os and Gy\'arf\'as \cite{EG97} implies
\[ \phi\left(n,k,\binom{k}{2}-\lfloor k/2 \rfloor +1\right) = \Omega\left(n^{4/3}\right). \]
That is, the boundary between a trivial problem and a non-trivial one passes between $\ell \ge \binom{k}{2}-\lfloor k/2 \rfloor +2$ and $\ell \le \binom{k}{2}-\lfloor k/2 \rfloor +1$.
(Recently, a stronger lower bound for this case was derived by Fish, Pohoata, and Sheffer \cite{FPS18}.)

Fox, Pach, and Suk \cite{FPS17} proved that for every $\eps>0$,
\begin{equation} \label{eq:FPS}
\phi\left(n,k,\binom{k}{2}-k+6\right) = \Omega\left(n^{8/7-\eps}\right).
\end{equation}
Moreover, Pohoata and Sheffer \cite{PS18} showed that for every $k>m\ge 2$,
\[ \phi\left(n,k,\binom{k}{2} - m\cdot \left\lfloor \frac{k}{m+1} \right\rfloor + m +1 \right) = \Omega\left(n^{1+ 1/m}\right). \]
For example, setting $m=2$ yields $\phi\left(n,k,\binom{k}{2} - 2\cdot \lfloor \frac{k}{3} \rfloor + 4 \right) = \Omega\left(n^{3/2}\right)$.
This bound is stronger than \eqref{eq:FPS} but holds for a smaller range of $\ell$.

No non-trivial upper bounds are known for $\ell =\binom{k}{2}-k+6$, for $\ell= \binom{k}{2} - m\cdot \left\lfloor \frac{k}{m+1} \right\rfloor + m +1$ with any $m\ge2$, or for any other similar values of $\ell$.
We conjecture that the bound $\phi\left(n,k,\binom{k}{2}-k+6\right)=\Omega\left(n^{2-\eps}\right)$ should hold for all of these cases.
This huge gap illustrates our current lack of understanding of the problem of distances with local properties.
By adapting a standard argument for double counting isosceles triangles, one obtains the bound $\phi\left(n,2k, \binom{2k}{2} - 2\cdot \binom{k}{2} + 1\right) = \Omega(n)$.
Once again, no non-trivial lower bounds are known for this case, which leads to an even larger gap.
More details about distances with local properties can be found in the survey \cite{Sheffer14}.

The goal of the current work is to provide a first non-trivial upper bound for the problem of distinct distances with local properties.
\begin{theorem}\label{th:DistUpper}
For every $n\ge k > 0$, we have 
\[ \phi\left(n,k,k^{\log_2 3}/2-1\right) =O\left(n^{\log_2 3}\right).\]
\end{theorem}

Note that $\log_2(3)\approx 1.585$.
Thus, Theorem \ref{th:DistUpper} provides a range of $\ell$ for which $\phi(n,k,\ell)$ is significantly smaller than $n^2$.
Unlike previous results for local properties problems, in Theorem \ref{th:DistUpper} the parameter $k$ is allowed to depend on $n$.

\parag{An Additive Combinatorics variant.}
Another variant of the local properties problem concerns sets of real numbers.
For a set $A\subset \RR$, we define the \emph{difference set} of $A$ as
\[ A-A = \left\{a-a' :\ a,a'\in A \right\}. \]

For parameters $n,k,\ell$, we consider sets $A$ of $n$ real numbers with the property that every subset $A'\subset A$ of size $k$ satisfies $|A'-A'| \ge \ell$.
Let $g(n,k,\ell)$ denote the minimum number of distinct distances that such a set of $n$ reals can span.
The known upper bounds for $g(n,k,\ell)$ are significantly stronger than the equivalent bounds for $\phi(n,k,\ell)$.
In particular, Fish, Pohoata, and Sheffer \cite{FPS18} derive a family of bounds such as
\begin{equation} \label{eq:ArithLower}  
g\left(n,k,2\cdot \binom{k}{2} - 4\cdot \binom{k/2}{2} +2\right) = \Omega\left(n^{2-\frac{8}{k}}\right). 
\end{equation}

Our upper bound construction is actually for $g(n,k,\ell)$.

\begin{theorem}\label{th:ArithUpper}
For every $n>k\ge 0$ such that $n$ is a power of two, we have
\[ g\left(n,k,k^{\log_2 3}\right) \le n^{\log_2 3}.\]
\end{theorem}

If $n$ is not a power of two, the proof of Theorem \ref{th:ArithUpper} still holds,  leading to the slightly weaker bound $g\left(n,k,k^{\log_2 3}\right) < 3n^{\log_2 3}$.

Note that $g(n,k,\ell)$ can be thought of as $\phi(n,k,\ell)$ when all of the points are on the $x$-axis.
In this case, every positive element of the difference set becomes a distance, and every non-positive element is ignored.
Thus, every upper bound for $g(n,k,\ell)$ is also an upper bound for $\phi(n,k,\ell/2-1)$.
In particular, Theorem \ref{th:DistUpper} is an immediate corollary of Theorem \ref{th:ArithUpper}.

To prove Theorem \ref{th:ArithUpper}, we construct a set $P$ of $n$ real numbers such that $|P-P|=n^{\log_2 3}$ and every small subset of $P$ spans many distinct differences.
Our set $P$ does not only satisfy the local property for one given $k$, but rather satisfies it simultaneously for every $k<n$.

Our approach is somewhat similar to a construction of points in $\RR^2$ of Erd\H os, Hickerson, and Pach \cite{EHP89}.
Moreover, it can be thought of as a generalization of Sidon sets (for example, see \cite{Obryant04}).
In a Sidon set every four elements span 13 distinct differences, while in our case every $k$ elements span at least $k^{\log_2 3}$ distinct differences.
Sidon sets consist only of integers, while our construction involves irrational numbers.
However, it is not difficult to revise our construction so that it uses only integers.
Our construction is also similar to Sidon sets in the sense that it spans few differences.

Some works on local properties problems, such as \cite{FPS18,PS18}, consider the difference set as containing only the positive differences.
This change does not affect the problem, since the asymptotic size of $A-A$ remains unchanged when moving from one definition to the other.
The current work relies on the standard definition of $A-A$, which should be kept in mind when comparing it to others.
For example, the bound \eqref{eq:ArithLower} looks slightly differently in \cite{FPS18} because of this difference of notation.

\parag{Acknowledgements.}
We would also like to thank Robert Krueger, Rados Radoicic, and Pablo Soberon for several helpful discussions.

\section{The construction} \label{sec:Const}

Let $n$ be a positive integer.
Consider positive irrational numbers $r_1,\ldots,r_n\in \RR$ such that the only integer solution to
\[ x_1\cdot r_1 + x_2\cdot r_2 + \cdots + x_{n} \cdot r_{n} = 0, \]
is $x_1=x_2=\cdots=x_{n}=0$.

For a set $A=\{a_1,\ldots,a_m\} \subset \RR$ and another number $b\in \RR$, we write
\[ b+A = \{b+a_1,b+a_2,\ldots,b+a_m\}. \]

For every integer $0\le j \le n$, we define a set $P_j$ of $2^j$ real numbers, as follows.
We set $P_0 = \{1\}$.
For $1\le j \le n$, we define $P_j = P_{j-1} \bigcup (r_j + P_{j-1})$.

The set $P_n$ can be thought of as a projection of the hypercube $\{0,1\}^n$ from $\RR^n$ onto a line.
In particular, this is the transformation that takes the $j$'th element of the standard basis of $\RR^n$ to $r_j$, and the origin to 1.
This is similar to a construction of Erd\H os, Hickerson, and Pach \cite{EHP89}, which is obtained by a projection of the $n$-dimensional hypercube onto a generic two-dimensional plane.

\begin{claim} \label{cl:DiffProps}
$|P_n - P_n| = |P_n|^{\log_2 3}$.
\end{claim}
\begin{proof}
Every difference in $P_n-P_n$ is of the form $x_1\cdot r_1 + x_2\cdot r_2 + \cdots + x_{n} \cdot r_{n}$ with $x_1,x_2,\ldots,x_n\in \{-1,0,1\}$.
By the definition of the $r_1,\ldots, r_n$, all of these differences are distinct, so $|P_n - P_n|=3^n$.
Recalling that $|P_n|=2^n$, we conclude that $|P_n - P_n| = 3^n = |P_n|^{\log_2 3}$.
\end{proof}

Set $p=\log_4 3$.
We require the following straightforward property.

\begin{claim} \label{cl:x^p}
Every positive $a,b\in\RR$ satisfy $a^p+b^p>(a+b)^p$.
\end{claim}
\begin{proof}
Without loss of generality, assume that $a\ge b$.
Let $f(x)= (a+x)^p + (b-x)^p$.
Then $f'(x) =  p(a+x)^{p-1} -p(b-x)^{p-1}$.
Since $p-1<0$, when $0< x<b$ we have that $(b-x)^{p-1} >(a+x)^{p-1}$, which in turn implies $f'(x)<0$.
Since $f(x)$ is continuous and decreasing in $(0,b)$, we obtain that $f(0)>f(b)$.
Plugging in the values of $f(0)$ and $f(b)$ yields the assertion of the claim.
\end{proof}

It remains to show that $P_n$ satisfies a local property.

\begin{lemma}\label{le:subset}
Let $Q$ be a subset of $P_n$ of size $k$.
Then $|Q-Q| \geq k^{\log_2 3}$.
\end{lemma}
\begin{proof}
Let $0 \le \ell \le n$ and let $t_1$ and $t_2$ be two reals of the form $y_{\ell+1}\cdot r_{\ell+1} + y_{\ell+2}\cdot r_{\ell+2} + \cdots + y_{n} \cdot r_{n}$ with $y_{\ell+1},y_{\ell+2},\ldots,y_n\in \{0,1\}$ (we allow the case where $t_1=t_2$).
When $\ell=n$, we have $t_1=t_2=0$.
Set $Q_1 = Q \cap (t_1+P_{\ell})$ and $Q_2 = Q \cap (t_2+P_{\ell})$.

We claim that
\begin{equation}\label{eq:mainInequality}
|Q_1 - Q_2| \geq (|Q_1| \cdot |Q_2|)^p.
\end{equation}
The statement of the lemma follows by applying \eqref{eq:mainInequality} with $\ell = n$ and $t_1=t_2=0$.
Indeed, in this case $|Q_1|=|Q_2|=k$, and $2p=\log_2 3$.
It thus remains to prove \eqref{eq:mainInequality}.

If either $|Q_1|= 0$ or $|Q_2| = 0$, then \eqref{eq:mainInequality} clearly holds.
We thus assume that $|Q_1|\ge 1$ and $|Q_2|\ge 1$.
We prove \eqref{eq:mainInequality} by induction on $\ell$.
For the induction basis, consider the case where $\ell=0$.
In this case we have that $|Q_1|= |Q_2| = 1$, and indeed $|Q_1 - Q_2| = 1 = (|Q_1| \cdot |Q_2|)^p$.

For the induction step, consider some $\ell>0$ and assume that \eqref{eq:mainInequality} holds for every smaller value of $\ell$.	
Recall that both $Q_1$ and $Q_2$ are contained in translated copies of $P_{\ell}$, and that $P_{\ell}= P_{\ell-1}\bigcup (r_{\ell} + P_{\ell-1})$.
We partition each of these translated copies of $P_{\ell}$ into the two corresponding copies of $P_{\ell-1}$, and separately consider the portion of $Q$ in each of the four resulting parts.
In particular, let
\begin{align*}
A &= Q_1 \cap (t_1+P_{\ell-1}), \\
B &= Q_1 \cap (t_1+r_\ell + P_{\ell-1}), \\
C &= Q_2 \cap (t_2+P_{\ell-1}), \\
D &= Q_2 \cap (t_2+r_\ell + P_{\ell-1}).
\end{align*}

Recall that every difference in $P_n-P_n$ is of the form $x_1\cdot r_1 + x_2\cdot r_2 + \cdots + x_{n} \cdot r_{n}$ with $x_1,x_2,\ldots,x_n\in \{-1,0,1\}$.
Note that for every difference in $(A-C)\cup(B-D)$ we have $x_\ell=0$.
Similarly, for every difference in $A-D$ we have $x_\ell=-1$, and for every difference in $B-C$ we have $x_\ell=1$.
By the definition of $\{r_1,\ldots,r_n\}$, we obtain that the three sets $(A-C)\cup(B-D), A-D, B-C$ are pairwise disjoint.
This in turn implies that
\[ |Q_1-Q_2|\ge |A-D| + |B-C| + \max\{|A-C|,|B-D|\}. \]

Set $a = |A|, b=|B|, c=|C|$, and $d=|D|$.
By applying the induction hypothesis with $\ell-1$, $t_1$, and $t_2+r_\ell$, we obtain that $|D-A|\ge (ad)^p$.
Similarly, the induction hypothesis implies that $|C-A|\ge (ac)^p$, $|C-B|\ge (bc)^p$, and $|D-B|\ge (bd)^p$.
Assume without loss of generality that $ac \geq bd$.
Then
\begin{equation}\label{ThreeTermsDiff}
|Q_1 - Q_2| \geq (ad)^p + (bc)^p + (ac)^p.
\end{equation}

To complete the induction step, we show that
\begin{equation}\label{eq:TightInequality}
(ac)^p + (ad)^p + (bc)^p \geq ((a+b)(c+d))^p.
\end{equation}
Combining this with \eqref{ThreeTermsDiff} would indeed imply $|Q_1 - Q_2|\geq ((a+b)(c+d))^p=(|Q_1|\cdot |Q_2|)^p$.

\parag{A technical calculation.}
By the assumption $|Q_1|\neq 0$, we cannot have $a=b=0$.
By the assumption $|Q_2|\neq 0$, we cannot have $c=d=0$.
If $a=c=0$, then we obtain \eqref{eq:mainInequality} by applying the induction hypothesis with $\ell-1$, $t_1+r_\ell$, and $t_2+r_\ell$.
We can symmetrically handle the cases of $a=d=0$, $b=c=0$, and $b=d=0$.
We may thus assume that at most one of $a,b,c,d$ is zero.

Since $ac\ge bd$, we get that $a\neq 0$ and that $c\neq 0$ (otherwise we would also have that $b=0$ or $d=0$).
If $b=0$ then \eqref{eq:TightInequality} becomes $(ac)^p + (ad)^p \geq (ac+ad)^p$.
By Claim \ref{cl:x^p}, we have that $(ac)^p + (ad)^p > (ac+ad)^p$ holds for every $a,c,d>0$.
We may thus assume that $b\neq 0$.
By a symmetric argument we may also assume that $d\neq 0$.

We set $\gamma = \frac{bd}{ac}$ and $x= d/c$.
Note that $0< \gamma \le 1$ and $x>0$.
Multiplying both sides of \eqref{eq:TightInequality} by $x^p/(ac)^p$ gives
\[ x^p + x^{2p} + \gamma^p \geq \left(x+x^2+\gamma+\gamma x\right)^p. \]
Rearranging this leads to
\[ x^p + x^{2p} + \gamma^p - (x+1)^p(x+\gamma)^p \ge 0. \]

Let $f(x,\gamma) = x^p + x^{2p} + \gamma^p - (x+1)^p(x+\gamma)^p$.
By the above, to prove \eqref{eq:TightInequality} it suffices to prove that $f(x,\gamma)\ge 0$ for every $0< \gamma \le 1$ and $x>0$.

We consider the first partial derivative
\[ \frac{\partial }{\partial \gamma} f(x,\gamma) = p\cdot \gamma^{p-1} - p\cdot (1+x)^p(\gamma + x)^{p-1}. \]
This expression is positive if and only if $\gamma^{p-1}> (1+x)^p(\gamma + x)^{p-1}$, or equivalently
\[\gamma < (1+x)^{p/(p-1)}(\gamma + x).  \]

Setting $r = (1+x)^{p/(p-1)}$, the above inequality becomes $\gamma < r(\gamma+x)$.
Rearranging this leads to
\[ \gamma < x\cdot \frac{r}{1-r}. \]

Since $p-1<0$ and $(1+x)^p>1$, we have that $0<r<1$.
This implies that $x\cdot \frac{r}{1-r} >0$.
Thus, for a fixed $x>0$, the function $f(x,\gamma)$ is increasing when $\gamma<x\cdot\frac{r}{1-r}$ and decreasing when
$\gamma>x\cdot\frac{r}{1-r}$.
We conclude that, to check that $f(x,\gamma)$ is non-negative when $0< \gamma\le 1$ and $x>0$, it suffices to consider the cases of $\gamma=0$ and $\gamma=1$.

Set $f_0(x)=f(x,0) = x^p + x^{2p} - (x^2+x)^p$.
Then $f_0'(x) = px^{p-1}+2px^{2p-1}-p(x^2+x)^{p-1}$.
Since  $x>0$ and $p-1<0$, we get that
\[ f_0'(x) > px^{p-1}-p(x^2+x)^p = p\left(x^{p-1}-\left(x^2+x\right)^{p-1}\right) > 0.\]
That is, $f_0(x)$ is an increasing function.
Since $f_0(0)=0$, we get that $f_0(x)$ is non-negative when $x\ge0$.

It remains to verify that $f(x,1)$ is non-negative for every $x>0$.
Set $f_1(x)=f(x,1) = x^p + x^{2p} +1 - (x+1)^{2p}$.
Then
\begin{align*}
f_1'(x) &= px^{p-1}+2px^{2p-1} -2p(x+1)^{2p-1}, \\
f_1''(x) &= p(p-1)x^{p-2}+2p(2p-1)x^{2p-2} -2p(2p-1)(x+1)^{2p-2}.
\end{align*}

We observe that $f_1(1) = 0$, $f'_1(1) =0$, and $f_1''(1) = 2p^2-3p/2 \approx 0.067$.
Thus, $x=1$ is a local minimum of $f_1(x)$.
By using Wolfram Mathematica, we found that the only roots of $f_1(x)$ are $x=0$ and $x=1$ (see below for more details).
Since $f_1(x)$ is a continuous function and $x=1$ is a local minimum, we conclude that $f_1(x)$ is non-negative when $x\ge 0$.

Since $f_1(x)$ is non-negative when $x\ge 0$, we conclude that \eqref{eq:TightInequality} holds.
This completes the proof of \eqref{eq:mainInequality}, and thus the proof of the lemma.

\parag{Precision issues.} One may rightly complain that Wolfram Mathematica could have missed another root of $f_1(x)$.
There seem to be two main issues that could potentially cause this: the range in which we look for potential zeros is an infinite interval (specifically, $[0,\infty)$), and a root might be missed in a region where the graph is very close to zero due to precision issues.
We now address both of these issues.

First, consider the special case where $x\ge 10$, and that there exists a positive integer $m$ satisfying $ac=bd=m$ and $ad=10m$.
Since $x=d/c$, we get that $d\ge 10c$.
Since $ac\ge bd$, we get that $a\ge 10b$.
This in turn implies that $ad\ge 10bd$.
Combining the above leads to
\[ (ac)^p + (ad)^p = (1+10^p)\cdot m^p > 12^pm^p = (ac+bd+ad)^p. \]
By Claim \ref{cl:x^p}, adding $(bc)^p$ to the left side of the inequality will cause a larger increase than adding $bc$ inside the parentheses on the right-hand side.
That is, in this special case, for any value of $bc$ we have \eqref{eq:TightInequality}.
Claim \ref{cl:x^p} also implies that after increasing the values of $ac$ and $ad$ by any amount, \eqref{eq:TightInequality} still holds.
That is, \eqref{eq:TightInequality} holds when $x\ge 10$, so there is no need to check for additional zeros of $f_1(x)$ in this case.

A symmetric argument holds when $x\le 1/10$.
In that case we have $c\ge 10d$, implying that $bc \ge 10bd$.
Thus, we are also not interested in zeros of $f_1(x)$ when $x\le 1/10$.
Combining the above, we get that it suffices to look for zero of $f_1(x)$ in the interval $(1/10,10)$.
That is, we no longer have the issue of finding zeros in an infinite interval.

It remains to verify that precision issues do not lead to missing roots in areas where $f_1(x)$ is close to zero.
A plot of $f_1(x)$ in the interval $(1/10,10)$ shows that the function is only close to zero around $x=1$.
Thus, missing roots due to precision issues can only occur in a small neighborhoods of $x=1$.

Recall that $x=1$ is a local minimum of $f_1(x)$ and that $f_1''(1) = 2p^2-3p/2 =2p^2-3p/2$.
The second order Taylor approximation of $f_1(x)$ at $x=1$ is
\[ F(x) = \frac{f_1''(1)}{2}(x-1)^2 = \left(p^2-\frac{3p}{4}\right)\cdot (x-1)^2 \approx 0.0337(x-1)^2. \]

To bound the error term of the approximation, we observe that
\begin{align*}
f_1^{(3)}(x) &= p(p-1)(p-2)x^{p-3} + 2p(2p-1)(2p-2)x^{2p-3} - 2p(2p-1)(2p-2)(x+1)^{2p-3} \\
&= p(p-1)\left((p-2)x^{p-3} + 4(2p-1)x^{2p-3} - 4(2p-1)(x+1)^{2p-3} \right).
\end{align*}
It is not difficult to verify that, when $0.6 \le x \le 2$, we have $\left|f_1^{(3)}(x)\right| < 0.1$.
Thus, the error of the second order approximation is $E(x)< \frac{0.1}{6}(x-1)^3$.
When $0.6\le x\le 2$ and $x\neq 1$, we have that $F(x) - E(x)> 0$, so the only root of $f(x)$ in $[0.6,2]$ is $x=1$.
This eliminates the precision issue around $x=1$.

It seems plausible that additional arguments, similar to the ones above, could completely eliminate the use of a computer program.
However, most likely such arguments would add several pages of tedious calculations.
We thus chose not to pursue this direction.
\end{proof}

Claim \ref{cl:DiffProps} and Lemma \ref{le:subset} immediately lead to Theorem \ref{th:ArithUpper}.

\begin{proof}[Proof of Theorem \ref{th:ArithUpper}.]
Let $n=2^m$ for some integer $m$.
Then the theorem holds by taking the set $P_m$ as described above.
Indeed, by Claim \ref{cl:DiffProps}, we have that $|P_m - P_m| = n^{\log_2 3}$.
By Lemma \ref{le:subset}, the set $P_m$ satisfies the required local property.
\end{proof}

\end{document}